\documentclass[11pt,hyp,]{amsart}
\usepackage{graphicx,amssymb,upref}

\let\<\langle
\let\>\rangle

\let\uml\"

\usepackage[english]{babel}
\usepackage{latexsym}
\usepackage{amssymb}
\usepackage{amscd}

\usepackage[all]{xy}
\usepackage{mathptmx}
\usepackage[cp850]{inputenc}
\usepackage[mathscr]{eucal}

\theoremstyle{plain}

\newtheorem{thm}{Theorem}[section]
\newtheorem{lema}{Lemma}[section]
\newtheorem{prop}{Proposition}[section]
\newtheorem{cor}{Corollary}[section]

\newcommand{\adef}{\begin{defn}}
\newcommand{\zdef}{\end{defn}}
\newtheorem{defn}[thm]{Definition}
\theoremstyle{remark}

\newcommand{\N}{\mathbb{N}}

\newcommand{\To}{\longrightarrow}
\def\PO{\operatorname{PO}}

\def\dens{\operatorname{dens}}
\def\dom{\operatorname{dom}}
\def\codom{\operatorname{cod}}
\newcommand{\aproof}{\begin{proof}}
\newcommand{\zproof}{\end{proof}}

\title{$1$-complemented subspaces\\ of Banach spaces of universal disposition}

\author{Jes\'us M.F. Castillo}
\address{Departamento de Matematicas,
Universidad de Extremadura, 06071 Badajoz, Espa\~{n}a} \email{castillo@unex.es}

\author{Yolanda Moreno}
\address{Departamento de Matematicas, Escuela Polit\'ecnica,
Universidad de Extremadura, 10003 C\'aceres, Espa\~{n}a} \email{ymoreno@unex.es}

\author{Marilda A. Sim\~{o}es}
\address{Dipartimento di Matematica "G. Castelnuovo",
Universit\'a di Roma "La Sapienza", P.le A. Moro 2,
00185 Roma, Italia}\email{simoes@mat.uniroma1.it}

\thanks{The research of the two first authors has been supported in part by project MTM2016-76958-C2-1-P and Ayuda a Grupos GR15152 de la Junta de Extremadura}

\begin{document}
\maketitle

\begin{abstract} We first unify all notions of partial injectivity appearing in the literature ---(universal) separable injectivity, (universal) $\aleph$-injectivity --- in the notion of $(\alpha, \beta)$-injectivity ($(\alpha, \beta)_\lambda$-injectivity if the parameter $\lambda$ has to be specified). Then, extend the notion of space of universal disposition to space of universal $(\alpha, \beta)$-disposition. Finally, we characterize the $1$-complemented subspaces of spaces of universal $(\alpha, \beta)$-disposition as precisely the spaces $(\alpha, \beta)_1$-injective.\end{abstract}

The purpose of this paper is to establish the connection between spaces (universally) $1$-separably injective and the spaces of universal disposition. A Banach space $E$ is said to be $\lambda$-\emph{separably injective} if
for every separable Banach space $X$ and each subspace $Y\subset
X$, every operator $t:Y\to E$ extends to an operator $T:X\to E$ with $\|T\|\leq \lambda\|t\|$. Chapter 3 considers the notion of universal disposition. Gurariy \cite{Gurariiold} introduced the property of a Banach space to be of universal disposition with respect to a given class $\mathfrak M$ of Banach spaces (or $\mathfrak M$-universal disposition, in short): $U$ is said to be of universal disposition for $\mathfrak M$ if given $A, B \in \mathfrak M$ and into isometries $u: A\to U$ and $\imath:A\to B$ there is an into isometry $u': B\to U$ such
that $u=u'\imath$. We are particularly interested in this section in the choices $\mathfrak M\in \{ \mathfrak F, \mathscr S \}$, where $\mathfrak F$ (resp $\mathscr S$) denote the classes of finite-dimensional (resp. separable) Banach spaces. \medskip

The paper \cite{accgm2} started the study of spaces of universal disposition after Gurariy \cite{Gurariiold}. All available results and several new were then presented in Chapter 4 of the monograph \cite{accgmLN}. Several overlooked and/or left open problems in the monograph were solved in \cite{castsimo}. An up-dated account of the facts known so far could be the following:

 \begin{itemize}
 \item Spaces of $\mathscr S$-universal disposition are $1$-separably injective \cite{accgm2} and \cite[Thm. 3.5]{accgmLN}.
 \item There are spaces of $\mathfrak F$-universal disposition that are not separably injective. Consequently, there are spaces of  $\mathfrak F$-universal disposition that are not of  $\mathscr S$-universal disposition \cite{castsimo}.
  \item $1$-separably injective are Lindenstrauss and Grothendieck spaces containing $c_0$  \cite{accgm1}.
 \item There are spaces of $\mathfrak F$-universal disposition containing complemented copies of $c_0$. Consequently, there are spaces of $\mathfrak F$-universal disposition that are not Grothendieck \cite{accgm2} and \cite[Prop. 3.16]{accgmLN}.
 \item Under {\sf CH}, $1$-separably injective spaces contain $\ell_\infty$ \cite{accgm1,accgmLN}. This can be no longer true without {\sf CH} \cite{aviko}.
  \end{itemize}

Our first aim is to connect the notions of $1$-separably injective space and space of universal disposition. We display the basic construction technique already used, in different versions, in \cite{accgm2,accgmLN,castsua,castsimo}:

\subsection*{The push-out construction} Recall first the push-out construction; which, given an isometry $u:A\to B$ and an operator
$t: A\to E$ will provides us with  an extension of $t$ through $u$ at the
cost of embedding $E$ in a larger space as it is showed in the diagram
$$
\begin{CD}
A @> u >> B\\
@Vt VV @VV t' V\\
E@> u' >> \PO
\end{CD}
$$
where $t'u=u't$. It is important to realize
that $u'$ is again an isometry and that $t'$ is a
contraction or an isometry if $t$ is.

\subsection*{The basic multiple push-out construction} To perform our basic construction we need the following data:
\begin{itemize}
\item A starting Banach space $X$.
\item A class $\mathfrak{M}$ of Banach spaces.
\item The family $\mathfrak{J}$ of all isometries acting between the elements of $\mathfrak{M}$.
\item A family $\mathfrak{L}$ of norm one $X$-valued operators defined on elements of $\mathfrak{M}$.
\item The smallest ordinal $\omega(\aleph)$ with cardinality $\aleph$.
\end{itemize}
For any operator $s : A\to B$, we establish $\dom(s) = A$ and $\codom (s) = B$. Notice that the codomain of an operator is usually larger than its range, and that the unique codomain of the elements of $\mathfrak{L}$ is $X$. Set $\Gamma=\{(u,t)\in \mathfrak{J}\times \frak L: \dom u=\dom t\}$ and
consider the Banach spaces of summable families $\ell_1(\Gamma,
\dom u)$ and $\ell_1(\Gamma, \codom u)$. We have an obvious
isometry
$$
\oplus\mathfrak{J} :\ell_1(\Gamma, \dom u)\To \ell_1(\Gamma, \codom
u)$$ defined by  $(x_{(u,t)})_{(u,t)\in\Gamma}\longmapsto
(u(x_{(u,t)}))_{(u,t)\in\Gamma} $; and a contractive operator
$$
\Sigma\frak L :\ell_1(\Gamma, \dom u)\To X,$$ given by
$(x_{(u,t)})_{(u,t)\in\Gamma}\longmapsto
\sum_{(u,t)\in\Gamma}t(x_{(u,t)})$. Observe that the notation is
slightly imprecise since both $ \oplus\mathfrak{J}$ and $ \Sigma\frak
L$ depend on $\Gamma$. We can form their push-out diagram
$$
\xymatrix{
\ell_1(\Gamma, \dom u)\ar[r]^{\oplus\mathfrak{J}} \ar[d]_{\Sigma\frak L} & \ell_1(\Gamma, \codom u)\ar[d]\\
X \ar[r]^{\imath_{(0,1)}} &\PO
}
$$
Thus, if we call $\mathcal S^0(X) = X$ then we have constructed the space
$\mathcal S^1(X) = \PO$ and an isometric enlargement of $\imath_{(0,1)}: \mathcal S^0(X)\to  \mathcal S^1(X) $ such that for every $t: A\to X$ in $\mathfrak L$, the operator $\imath t$ can be
extended to an operator $t':B \to \PO$ through any embedding
$u:A\to B$ in $\mathfrak{J}$ provided $\dom u=\dom t= A$. \medskip

In the next step we keep the family $\mathfrak{J}$ of isometries, replace the starting space $X$ by $\mathcal S^1(X)$
and $\frak L$ by a family $\mathfrak L_1$ of norm one operators $\dom u \to \mathcal S^1(X)$, $u\in \mathfrak J$, and proceed
again. More precisely, the inductive step is as follows. Suppose we have constructed the directed system $(\mathcal S^\alpha(X))_{\alpha<\beta}$ for all ordinals $\alpha<\beta$, including the corresponding linking maps $\imath_{(\alpha,\gamma)}: \mathcal S^\alpha(X)\To \mathcal S^\gamma(X)$ for $\alpha<\gamma<\beta$. To define $\mathcal S^\beta(X)$ and the maps $\imath_{(\alpha,\beta)}: \mathcal S^\alpha(X) \To \mathcal S^\beta(X)$ we consider separately two cases, as usual: if $\beta$ is a limit ordinal, then we take $\mathcal S^\beta(X)$ as the direct limit of the system $(\mathcal S^\alpha(X))_{\alpha<\beta}$ and
$\imath_{(\alpha,\beta)}: \mathcal S^\alpha(X) \To \mathcal S^\beta(X)$ the natural inclusion map. Otherwise $\beta=\alpha + 1$ is a successor ordinal and  we construct $\mathcal S^{\beta}(X)$ applying the push-out construction as above with the following data:  $\mathcal S^\alpha(X)$ is the starting space, $\frak
J$ keeps being the set of all isometries acting between the elements of
$\mathfrak M$ and
$\frak L_\alpha$ is the family of all isometries $t:S\to \mathcal S^\alpha(X)$, where $S\in \mathfrak M$. We then set $\Gamma_\alpha=\{(u,t)\in \frak J\times \frak L_\alpha:
\dom u=\dom t\}$ and make the push-out
\begin{equation}\label{see0}
\begin{CD}
\ell_1(\Gamma_\alpha, \dom u)@> \oplus\frak I_\alpha >> \ell_1(\Gamma_\alpha, \codom u)\\
@V \Sigma\frak L_\alpha VV @VVV\\
\mathcal S^\alpha (X)@>>> \PO
\end{CD}
\end{equation}
thus obtaining ${\mathcal S}^{\alpha + 1}(X)=\PO$.
The embedding $\imath_{(\alpha,\beta)}$ is the lower arrow in the above diagram; by composition with $\imath_{(\alpha,\beta)}$ we get the embeddings $\imath_{(\gamma,\beta)} =\imath_{(\alpha, \beta)}\imath_{(\gamma, \alpha)}$, for all $\gamma < \alpha$.\medskip

\begin{prop}\label{coimes} A Banach space is $1$-separably injective if and only if it is a $1$-complemented subspace of a space of $\mathscr S$-universal disposition.
\end{prop}
\begin{proof}  Since spaces of $\mathscr S$-universal disposition are $1$-separably injective \cite{accgm2}, their $1$-complemented subspaces also are $1$-separably injective.
To show the converse, let $\Theta$ be a $1$-separably injective space.
Set now the following input data in the basic construction:
\begin{itemize}
\item $X=\Theta$.
\item $\mathfrak M = \mathscr S$
\item The family $\mathfrak{J}$ of all isometries acting between the elements of $\mathscr S$.
\item The family $\mathfrak{L}$ of into $\Theta$-valued isometries defined on elements of $\mathscr S$.
\item The first uncountable ordinal $\omega_1$.
\end{itemize}

From the construction we get the space $\mathcal S^{\omega_1}(\Theta)$. This space is of $\mathscr S$-universal disposition (cf. \cite[Prop. 3.3]{accgmLN}).\medskip

\noindent \textbf{Claim. If $\Theta$ is $1$-separably injective then it is $1$-complemented in $\mathcal S^{\omega_1}(\Theta)$}.\medskip

\noindent\emph{Proof of the Claim.} It is clear operators from separable spaces into $\Theta$ extend with the same norm to separable superspaces. Thus occurs to the operators forming $\Sigma \mathfrak L$. Pick the identity $1_\Theta$ and observe that the
composition $1_\Theta \Sigma \mathfrak L$ extends to an operator $\ell_1(\Gamma, \codom
u) \to \Theta$; hence, by the push-out property, there is an operator
 $\mathcal S^1(\Theta)\to \Theta$ that extends $1_\Theta$ through $\imath_1: \Theta \to  \mathcal S^1(\Theta)$. In other words, $\Theta$ is $1$-complemented in  $\mathcal S^1(\Theta)$. The argument can be transfinitely iterated to conclude that $\Theta$ is $1$-complemented in  $\mathcal S^{\omega_1}(\Theta)$.\end{proof}

A basic form of Theorem \ref{coimes} appears in \cite[Thm. 3.22]{kubis} as:\emph{ Assume the continuum hypothesis. Let V be the unique Banach
space of density $\aleph_1$ that is of universal disposition for separable spaces. A Banach
space of density $\leq \aleph_1$ is isometric to a 1-complemented subspace of V if and only
if it is 1-separably injective.}\medskip

The paper \cite{accgm4} (see also \cite[Chapter 5]{accgmLN}) contains  the higher cardinal generalization  of the notion of separable injectivity as follows. A Banach space $E$ is said to be \emph{$\aleph$-injective} if for every Banach space $X$ with
density character $<\aleph$ and each subspace $Y\subset X$ every operator $t:Y\to E$ can be extended
to an operator $T:X\to E$. When for every operator $t$ there exists some extension $T$ such that $\|T\|\leq\lambda\|t\|$ we say that $E$
is \emph{$(\lambda,\aleph)$-injective}.\medskip

The case $\aleph = \aleph_1$ corresponds to separable injectivity and thus the resulting name for separable injectivity is $\aleph_1$-injectivity (not $\aleph_0$-injectivity),
which is perhaps surprising. Nevertheless, we have followed the uses of set theory where properties labeled by a cardinal
$\aleph$ always indicate that something happens for sets whose cardinality is strictly lesser
than $\aleph$. Let $\mathscr S_\aleph$ the class of all Banach spaces having density character less than $\aleph$ (thus, $\mathscr S = \mathscr S_{\aleph_1}$). The general version of Proposition \ref{coimes} is

\begin{prop}\label{1asi} A Banach space is $(1, \aleph)$-injective if and only if it is a $1$-complemented subspace of a space of $\mathscr S_\aleph$-universal disposition.
\end{prop}
\begin{proof} The proof that spaces of $\mathscr S_\aleph$-universal disposition are $(1, \aleph)$-injective is the same as that for $\mathscr S$ (\cite{accgm2} and \cite[Thm. 3.5]{accgmLN}); thus, their $1$-complemented subspaces also are $(1, \aleph)$-injective. To show the converse, let $\Theta$ be a $(1, \aleph)$-injective space. Set now the input data as:
\begin{itemize}
\item $X=\Theta$.
\item $\mathfrak M = \mathscr S_\aleph$
\item The family $\mathfrak{J}$ of all isometries acting between the elements of $\mathscr S_\aleph$.
\item The family $\mathfrak{L}$ of into $\Theta$-valued isometries defined on elements of $\mathscr S_\aleph$.
\item The ordinal $\omega(2^\aleph)$.
\end{itemize}

\noindent \textbf{Claim. The space $\mathcal S^{\omega(2^\aleph)}(X)$ is a space of  $\mathscr S_\aleph$-universal disposition.}\medskip

\noindent\emph{Proof of the Claim.} Let $A, B \in \mathscr S_\aleph$ and $i: A\to B$ an into isometry. Let $j:A\to X$ an into isometry.
Since the cofinality of $2^\aleph$ is greater than $\aleph$, the image of $j$ must fall into some $\mathcal S^\alpha(X)$ for some $\alpha<2^\aleph$. Now, since both $j$ and $i$ are part of the amalgam of operators respect to with one does push out, there is an into isometry $j': B\to \mathcal S^{\alpha+1}(X)$ extending $j$. In particular, there is an into isometry $j': B\to \mathcal S^{\omega(2^\aleph)}(X)$ extending $j$. \hfill{$\square$}\medskip

\noindent \textbf{Claim. If $\Theta$ is $(1, \aleph)$-injective then it is $1$-complemented in $\mathcal S^{\omega(2^\aleph)}(\Theta)$}\medskip

The proof is exactly as the one we did for $\mathscr S$. This concludes the proof of the Proposition.\end{proof}

We want now to handle a different set of ideas. The papers \cite{accgm1,accgmLN} also introduce the notion of universally separably injective space as follows: a Banach space $E$ is said to be \emph{universally $\lambda$-separably injective} if for every Banach space $X$ and each separable subspace $Y\subset
X$, every operator $t:Y\to E$ extends to an operator $T:X\to E$ with $\|T\|\leq \lambda\|t\|$. It turns out that there are examples of universally separably injective spaces appearing in nature, such as:

\begin{itemize}
\item The Pe\l czy\'nski-Sudakov \cite{pelisuda} spaces $\ell_\infty^{\aleph}(\Gamma)$.
\item The space $\ell_\infty/c_0$ and, in general, quotients of injective spaces by separably injective spaces. \cite{accgm1}
\item Ultrapowers of $\mathcal L_\infty$-spaces with respect to countably incomplete ultrafilters on $\N$. \cite[Thm.4.4]{accgmLN}
\end{itemize}

The higher cardinal generalization of this notion was studied in \cite{accgm4} (see also \cite[Chapter 5]{accgmLN}). The space $E$ is said to be \emph{universally $\aleph$-injective} if for every space $X$ and
each subspace $Y\subset X$ with density character $< \aleph$, every operator $t:Y\to E$ can be extended
to an operator $T:X\to E$.  When for every operator $t$ there exists some extension $T$ such that $\|T\|\leq\lambda\|t\|$ we say that $E$
is universally \emph{ $(\lambda,\aleph)$-injective}. To obtain analogues to Propositions \ref{coimes} and \ref{1asi} we need to find a common generalization of the notions of $\aleph$-injectivity and universal $\aleph$-injectivity and then adapt the notion of universal disposition.\medskip

\noindent \textbf{Definition.} Let $(\alpha, \beta)$ two infinite cardinals, $\alpha\leq \beta$. The space $E$ is said to be $(\alpha, \beta)$-injective if every
operator $t: A\to E$  from a subspace  $A\subset B$ with density character $< \alpha$ can be extended
to an operator $T: B\to E$.  When for every there exists some extension $T$ such that $\|T\|\leq\lambda\|t\|$ we say that $E$
is $(\alpha, \beta)_\lambda$-injective.\medskip

Observe that $\aleph$-injectivity corresponds to $(\aleph, \aleph)$-injectivity while universal $\aleph$-injectivity corresponds to $(\aleph, 2^\aleph)$-injectivity since
every space with density character $\aleph$ can be embedded into an injective space with density character  $2^\aleph$.\medskip

\noindent \textbf{Definition.} Let $(\alpha, \beta)$ two infinite cardinals, $\alpha\leq \beta$. The space $E$ is said to be of $(\alpha, \beta)$-universal disposition
if given spaces $A, B$ with $\dens A<\alpha$ and  $\dens B<\beta$  and into isometries $u: A\to E$ and $\imath:A\to B$ there is an into isometry $u': B\to E$ such
that $u=u'\imath$. \medskip

One has

\begin{thm} A Banach space is $(\alpha, \beta)_1$-injective if and only if it is a $1$-complemented subspace of a space of $(\alpha, \beta)$-universal disposition.
\end{thm}
\begin{proof} The proof is a variation of those of Propositions \ref{coimes} and \ref{1asi}. The three claims involved are:\medskip

\textbf{Claim. Spaces of $(\alpha, \beta)$-universal disposition are $(\alpha, \beta)_1$-injective}\medskip

with the obvious proof as before. Therefore, $1$-complemented subspaces of spaces of $(\alpha, \beta)$-universal disposition also are $(\alpha, \beta)_1$-injective. To show show the converse, let $\Theta$ be a  $(\alpha, \beta)_1$-injective space. Set now the input data as:
\begin{itemize}
\item $X=\Theta$.
\item $\mathfrak M = \mathscr S_\alpha$
\item The family $\mathfrak{J}$ of all into isometries acting from spaces in $\mathscr S_\alpha$ into spaces in $\mathscr S_\beta$.
\item The family $\mathfrak{L}$ of into $\Theta$-valued isometries defined on elements of $\mathscr S_\alpha$.
\item The ordinal $\omega(2^\alpha)$.
\end{itemize}
Let us call $\mathcal S^{\alpha, \beta}(\Theta)$ the resulting space $\mathcal S^{\omega(2^\aleph)}(\Theta)$.\medskip

\noindent \textbf{Claim. For any choice of $X$, the space $\mathcal S^{\alpha, \beta}(X)$ is a space of  $(\alpha, \beta)$-universal disposition.}\medskip

\noindent\emph{Proof of the Claim.} Let $A \in \mathscr S_\alpha$, $B \in \mathscr S_\beta$ and  $i: A\to B$ an into isometry. Let $j:A\to X$ an into isometry.
Since the cofinality of $2^\alpha$ is greater than $\alpha$, the image of $j$ must fall into some $\mathcal S^\gamma(X)$ for some $\gamma <2^\alpha$. Now, since both $j$ and $i$ are part of the amalgam of operators respect to with one does push out, there is an into isometry $j': B\to \mathcal S^{\gamma+1}(X)$ extending $j$. In particular, there is an into isometry $j': B\to \mathcal S^{\alpha, \beta}(X)$ extending $j$.\hfill{$\square$}\medskip

\noindent \textbf{Claim. If $\Theta$ is $(\alpha, \beta)_1$-injective then it is $1$-complemented in $\mathcal S^{\alpha, \beta}(\Theta)$}\medskip

\noindent\emph{Proof of the Claim.} It is clear operators from $\mathscr S_\alpha$ into $\Theta$ extend with the same norm to superspaces in $\mathscr S_\beta$. Thus occurs to the operators forming $\Sigma \mathfrak L$. Pick the identity $1_\Theta$ and observe that the
composition $1_\Theta \Sigma \mathfrak L$ extends to an operator $\ell_1(\Gamma, \codom
u) \to \Theta$; hence, by the push-out property, there is an operator
 $\mathcal S^1(\Theta)\to \Theta$ that extends $1_\Theta$ through $\imath_1: \Theta \to  \mathcal S^1(\Theta)$. In other words, $\Theta$ is $1$-complemented in  $\mathcal S^1(\Theta)$. The argument can be transfinitely iterated to conclude that $\Theta$ is $1$-complemented in  $\mathcal S^{\alpha, \beta}(\Theta)$ \hfill{$\square$}

 This concludes the proof of the Theorem.  \end{proof}

The spaces  $\mathcal S^{\alpha, \beta}(X)$ are moreover the first examples of spaces of $(\alpha, \beta)$-universal disposition for $\alpha\neq \beta$. Replacing isometries by isomorphisms one would get the very different notion of space $(\alpha, \beta)$-automorphic (see \cite{accgmLN}). It is straightforward that every Banach space of $(\alpha, \beta)$-universal disposition is $(\alpha, \beta)$-automorphic ; while the converse fails: A good classical example is the space $\ell_\infty$, which  has the property that every isomorphism between two separable subspaces can be extended to an automorphism of $\ell_\infty$  \cite{lindrose} (see also \cite[Prop. 2.5.2]{accgmLN}): or, in other words, it is $(\aleph_1, \mathfrak c^+)$-automorphic. However, it is not of $(\aleph_1, \mathfrak c^+)$-universal disposition since not every isometry between two finite dimensional subspaces can be extended to an isometry of $\ell_\infty$ (see \cite{accgm2} and  \cite[Section 3.3.4]{accgmLN})\medskip

To get our analysis complete it only remains to analyze the case $\aleph_0$. The spaces of universal $\aleph_0$-disposition are obviously the spaces of  $\mathfrak F$-universal disposition; while the spaces $\aleph_0$-injective are called in \cite{cgk} and \cite[Def.1.3]{accgmLN} locally injective spaces and shown to be the $\mathcal L_\infty$-spaces.
The meaning of \emph{locally $\lambda$-injective} should be clear; as well as the fact that $\mathcal L_{\infty, \lambda}$-spaces are locally $\lambda$-injective. The following result should be clear by now:

\begin{prop} A Banach space is locally $1$-injective if and only if it is a $1$-complemented subspace of a space of $\mathfrak F$-universal disposition.
\end{prop}

The list of locally $1$-injective spaces however is not simple: $1$-separably injective spaces are obviously locally $1$-injective, but $c_0$, who is only $2$-separably injective, is also locally $1$-injective. The case of $c_0$ can be generalized. Recall that a Banach space is said to be polyhedral if the closed unit ball of every finite dimensional subspace is the closed convex hull of a finite number of points. Then every Lindenstrauss polyhedral space is locally $1$-injective as it follows from \cite[Prop. 7.2]{lindmem}.
However, spaces of $\mathfrak  F$-universal disposition cannot be polyhedral since they contain isometric copies of all separable spaces. On the other hand, even $c$, which is isomorphic to $c_0$, is not locally $1$-injective: to show this, observe that, in the spirit of Lindenstrauss \cite{lindmem} one has:

\begin{lema} The following properties are equivalent:
\begin{enumerate}
\item $X$ is locally $1$-injective.
\item Let $F\subset G$ be Banach spaces so that $\dim G/F<+\infty$. Every finite rank operator $\tau: F\to X$ can be extended to an operator $T:G\to X$ with the same norm.
\item Let $F\subset G$ be Banach spaces so that $\dim G/F<+\infty$. Compact operators $\tau: F\to X$ can be extended to operators $T: G\to X$ with the same norm.
\end{enumerate}
\end{lema}
\begin{proof} To show that (1) implies (2), consider the push-out diagram

$$\begin{CD}
F@>>> G @>>> G/F\\
@V{\tau}VV @VV{\tau'}V  @|\\
\tau(F) @>>> \PO @>>> G/F\\
@V{\imath}VV\\
X
\end{CD}$$

where $\imath$ is the canonical inclusion. The local $1$-injectivity allows to extend $\imath$ to a norm one operator $I: \PO\to X$ since $\dim \PO/\tau(F) = \dim G/F <+\infty$, which shows that $\PO$ is finite-dimensional. Therefore $I\tau'$ is an extension of $\tau$ with the same norm. We show now that (2) implies (3): as we have already mentioned, locally injective spaces are $\mathcal L_\infty$-spaces, end thus they have the Bounded Approximation Property, which makes any compact operator approximable by finite rank operators. Now, if every finite rank operator admits an equal norm extension, every operator that is in the closure of finite-rank operators also admits an equal norm extension. Finally, that (3) implies (1) is obvious.
\end{proof}

\begin{cor} The space $c$ is not locally $1$-injective
\end{cor}
\begin{proof} This is consequence of an example in \cite{johnzipproc} of a compact operator $H\to c$ that cannot be extended with the same norm to one more dimension.\end{proof}
\newpage

To extend the results in this paper to the category of $p$-Banach spaces, a few items have to be taken into account. On one hand, the multiple push-out construction requires only minor adjustments to work (these can be seen explicitly in \cite{cgk}) while the notions of $p$-Banach space universal disposition and separably injective can be translated verbatim. The fact that a space of universal disposition for separable $p$-Banach spaces is also separably injective is straightforward and can be found in \cite[Prop. 3.50 (2)]{accgmLN}.
Since, as we have already shown, spaces of $(\alpha, \beta)$-universal disposition are $(\alpha, \beta)_1$-injective, one faces an apparent contradiction with
the fact that no universally separably injective $p$-Banach spaces exist at all (cf. \cite[Prop. 3.45]{accgmLN}. The key point is that $(\alpha, 2^\alpha)$-injective $p$-Banach spaces are no longer universally $\alpha$-injective since no injective $p$-Banach spaces exist at all. It is however still true that

\begin{thm} A $p$-Banach space is $(\alpha, \beta)_1$-injective if and only if it is a $1$-complemented subspace of a space of $(\alpha, \beta)$-universal disposition.
\end{thm}

\end{document}